\documentclass[11pt]{article}
\usepackage{amsmath}
\usepackage{amssymb}
\usepackage{amsfonts}
\usepackage{graphicx}

\newtheorem{theorem}{Theorem}

\newtheorem{corollary}[theorem]{Corollary}

\newtheorem{definition}[theorem]{Definition}
\newtheorem{example}[theorem]{Example}

\newenvironment{proof}[1][Proof]{\noindent \textbf{#1.} }{\  \rule{0.5em}{0.5em}}

\begin{document}

\title{A new class of curves generalizing helix and rectifying curves}
\author{\textsc{Fouzi Hathout} \\
Department of Mathematics, University\ of \textsc{Sa\"{\i}da},\\
20000 \textsc{Sa\"{\i}da}, \textsc{Algeria}.\\
Email: fouzi.hathout@univ-saida.dz, f.hathout@gmail.com}
\date{}
\maketitle

\begin{abstract}
In this paper, we introduce a new class of curves $\alpha $ called a $f$%
-rectifying curves, which its $f$-position vector defined by $\alpha
_{f}(s)=\int f(s)T(s)ds$ always lie in the rectifying plane of $\alpha $,
where $f$\ is an integrable function and $T$ is the speed curve of $\alpha $%
. In particular case, when the function $f\equiv 0$ or constant, the class
of $f$-rectifying curves are helix or rectifying curves, respectively. The
classification and the characterization of such curves in terms of their
curvature and the torsion functions are given with a physical
interpretation. We close this study with some examples.


\end{abstract}

\section{Introduction}

Let $\mathbb{E}^{3}$ be an Euclidean 3-space, we denote by $<x,y>$ the
standard inner product for any arbitrary vectors $x$ and $y$ in $\mathbb{E}%
^{3}$. The norm of $x$ is denoted by $\left \vert x\right \vert =\sqrt{<x,x>}
$.

Let $\alpha :I\subset \mathbb{R}\rightarrow \mathbb{E}^{3}$ be a non null
speed curve. The arc-length parameter $s$ of a curve $\alpha $ is determined
such that $\left \vert \alpha ^{\prime }(s)\right \vert =\left \vert
T(s)\right \vert =1$. We define the curvature function of $\alpha $ by $%
\kappa (s)=\left \vert T^{\prime }(s)\right \vert .$ If $\kappa (s)\neq 0,$
then the unit principal normal vector $N(s)$ of the curve $\alpha $ at $s$
is given by $\alpha ^{\prime \prime }(s)=T^{\prime }(s)=\kappa (s)N(s).$ The
binormal vector is $B(s)=T(s)\times N(s)$ (the symbol $\times $ is vector
product).

The Frenet-Serret formulas are%
\begin{equation}
T^{\prime }=\kappa N;\  \ N^{\prime }=-\kappa T+\tau B;\  \ B^{\prime }=-\tau N
\label{1}
\end{equation}%
where the function $\tau (s)$ is the torsion function of $\alpha $ at $s$. A
curve is called a \textit{twisted} curve if has non zero curvature and
torsion. The planes spanned by $\{T,N\}$, $\{T,B\}$, and $\{N,B\}$ are
called the \textit{osculating plane}, \textit{the rectifying plane}, and
\textit{the normal plane}, respectively.

We keep the name helix for a curve $\alpha $ in $\mathbb{E}^{3}$ if its
tangent vector $T$ makes a constant angle with a fixed direction $X$ called
also the axis. The vector $X$ along helix curve lies in the rectifying
plane, it can be given by
\begin{equation}
X=\cos \theta \ T+\sin \theta \ B  \label{2}
\end{equation}%
here $\theta $ is a constant angle different from $\frac{\pi }{2}$ (see \cite%
{bm}).\newline
In \cite{Cb}, the author introduce a rectifying curves, as space curves $%
\alpha $ whose position vector always lie in its rectifying plane. The
position vector $\alpha (s)$ of a rectifying curve satisfies%
\begin{equation}
\alpha (s)=(s+a)T+bB  \label{3}
\end{equation}%
where $a$ and $b$ are some real constants.\newline
In the terms of curvature and torsion, a curve $\alpha $ is a general helix
and congruent to a rectifying curve if and only if the ratio of torsion to
curvature is
\begin{equation}
\frac{\tau }{\kappa }=c_{1}\text{ (constant)}  \label{2.1}
\end{equation}%
and%
\begin{equation}
\frac{\tau }{\kappa }=c_{2}s+c_{3}\text{ (linear function)},  \label{3.1}
\end{equation}%
respectively, where $c_{1,2}$ is non null constants and $c_{3}$ is a
constant.\newline
Therefore, the rectifying plane of a curve $\alpha $ play an important role
to this two classes of curves (i.e. helix and rectifying).

Motivated by above definitions of helix and rectifying curves given in Eq.(%
\ref{2} and \ref{3}), and their characterizations in the terms of curvature
and torsion given in Eq.(\ref{2.1} and \ref{3.1}), it is natural to ask the
following geometric question: Is there a class of curves generalizing the
classes of helix and rectifying curves?

Firstly, let's define a new vector $\alpha _{f}$ that we call'it $f$\textit{%
-position vector }of the curve\textit{\ }$\alpha $ by
\begin{equation*}
\alpha _{f}(s)=\int f(s)T(s)ds
\end{equation*}%
where $f$ is an integrable function. By a simple calculate, for $f\equiv 0$
we find the right side of Eq.(\ref{2}) and for non null constant function $%
f, $ we get, up to parametrization, right side of the Eq.(\ref{3}).\newline
When the $f$\textit{-position vector }$\alpha _{f}$ lie in the rectifying
plane of $\alpha $ i.e.%
\begin{equation}
\alpha _{f}(s)=\int f(s)T(s)ds=\lambda (s)T+\mu (s)B  \label{6}
\end{equation}%
for $f(s)\equiv 0$ and $f(s)\equiv f$ (constant), we find helix and
rectifying definitions, respectively, where $\lambda $ and $\mu $ are some
functions. The generalization in the terms of curvature and the torsion
functions of the Eqs(\ref{2.1} and \ref{3.1}) will be presented in the
Theorem \ref{T2}.

\bigskip

Now, we are able to introduce the following definition about such curves.

\begin{definition}
\label{D1}Let $\alpha :I\subset \mathbb{R}\rightarrow \mathbb{E}^{3}$ be a
curve with Frenet apparatus $\{T,N,B,\kappa ,\tau \}$ and $f$ be an
integrable function in parameter $s.$ We call the curve $\alpha $ a $f$%
-rectifying curve if its $f$\textit{-position vector }$\alpha _{f}$ lie
always in the rectifying plane of $\alpha $ i.e.%
\begin{equation*}
\alpha _{f}(s)=\lambda (s)T+\mu (s)B
\end{equation*}%
where $\lambda (s)$ and $\mu (s)$ are some real functions$.$
\end{definition}

Consequently, the Definition \ref{D1} coincides with helix curve or
rectifying curve definitions when the function $f$ is a null or a constant,
respectively.\newline
Hence, when the function $f$ varies in the set of all integrable functions,
the $f$-position vector $\alpha _{f}$ give an enlarged determination of $%
\alpha $ and the class of $f$-rectifying curves present a generalization of
helix and rectifying space curves.

\bigskip

The paper is organized as follow;\newline
In the section 2, we give characterizations of $f$-rectifying curves by the
Theorem \ref{T1}. In section 3, we prove that a twisted curve is congruent
to a $f$-rectifying curve if and only if the ratio $\tau /\kappa $ is the
primitive function $F$ of $f.$ We also give a physical signification of $f$%
-rectifying curves in mechanics terms. The end section is devoted to the
determination explicitly of all $f$-rectifying curves and we close this
study with some examples.

\section{Characterization of $f$-rectifying curves}

For the characterizations of $f$-rectifying curves, we have the following
theorem

\begin{theorem}
\label{T1}Let $\alpha :I\subset \mathbb{R}\rightarrow \mathbb{E}^{3}$ be a $%
f $-rectifying ($f$ is nonzero function) curve with strictly positive
curvature function and $s$ be its arclength. Then\newline
\textbf{1.} The norm function $\rho (s)=\left \vert \alpha
_{f}(s)\right
\vert $ satisfies%
\begin{equation*}
\rho (s)=\sqrt{F^{2}(s)+c^{2}}
\end{equation*}%
where $F$ is the primitive function of $f$ and $c$ is a non null constant.%
\newline
\textbf{2.} The tangential component of the $f$-position vector $\alpha _{f}$
is
\begin{equation*}
\left \langle \alpha _{f},T\right \rangle =F(s)
\end{equation*}%
\textbf{3.} The normal component of the $f$-position vector $\alpha _{f}$
has constant length.\newline
\textbf{4.} The torsion $\tau $ is nonzero, and the binormal component of
the $f$-position vector $<\alpha _{f},B>$\ is constant.\newline
Conversely, if $\alpha :I\subset \mathbb{R}\rightarrow \mathbb{E}^{3}$ is a
curve with a positive curvature $\kappa $ and if one of the assertions
\textbf{1}, \textbf{2}, \textbf{3} or \textbf{4} holds, then $\alpha $ is a $%
f$-rectifying curve.
\end{theorem}

\begin{proof}
Let $\alpha :I\subset \mathbb{R}\rightarrow \mathbb{E}^{3}$ be a $f$%
-rectifying curve parameterized by arclength $s,$ we suppose that $\alpha $
is non helix curve (i.e. $f$ nonzero function). From the Definition \ref{D1}%
, we have%
\begin{equation*}
\alpha _{f}(s)=\int f(s)d\alpha =\lambda (s)\ T+\mu (s)\ B
\end{equation*}%
Differentiating the Eq.(\ref{6}) with respect to $s$ and using the Frenet
formulas Eq.(\ref{1}), we get%
\begin{equation*}
f(s)T=\lambda ^{\prime }(s)\ T+\left( \lambda (s)\kappa -\mu (s)\tau \right)
N+\mu ^{\prime }(s)\ B
\end{equation*}%
by comparing, we have%
\begin{equation}
\left \{
\begin{array}{l}
\lambda (s)=\int f(s)ds=F(s) \\
\dfrac{\tau }{\kappa }=\dfrac{\lambda (s)}{\mu (s)}, \\
\mu (s)=\mu \text{ non null constant}%
\end{array}%
\right.  \label{9}
\end{equation}%
\textbf{1.} We have for the norm function%
\begin{eqnarray}
\rho ^{2\prime }(s) &=&2\left \langle \alpha _{f}(s),f(s)T\right \rangle
=2\left \langle F(s)T+\mu B,f(s)T\right \rangle  \label{9.1} \\
&=&2F(s)f(s)=F^{2^{\prime }}(s)\  \text{and}  \notag \\
\rho ^{2}(s) &=&F^{2}(s)+\bar{c}
\end{eqnarray}%
Because $0$ is in $I$, we must have $\bar{c}>0$, then%
\begin{equation*}
\rho (s)=\sqrt{F^{2}(s)+c^{2}}
\end{equation*}%
where $\bar{c}=c^{2}$.\newline
\textbf{2.} It's a direct consequence from the equations Eq.(\ref{6}) and
Eq.(\ref{9})$_{(1)}$.\newline
\textbf{3.} Let us put $\alpha _{f}(s)=m(s)T(s)+\alpha _{f}^{N}(s)$, where $%
m(s)$ is arbitrary differentiable function. Comparing with the Eq. (\ref{6}%
), we conclude that $\alpha _{f}^{N}(s)=\mu B,$ and $\left \langle \alpha
_{f}^{N}(s),B\right \rangle =\mu $ is a constant from Eq.(\ref{9})$_{(3)}$
then this yields assertion (3).\newline
\textbf{4.} We can easily get (4) from Eq.(\ref{9})$_{(1,2)}$ and the fact
that $\kappa >0$.\newline
Conversely, \newline
Suppose that the assertions (1) or (2) holds. Then we have $\left \langle
\alpha _{f}(s),T\right \rangle =F(s)$ and by taking the derivative of the
last equation with respect to $s$, we get $\kappa \left \langle \alpha
_{f}(s),N\right \rangle =0$. Taking account that $\kappa >0$, we have $%
\left
\langle \alpha _{f},N\right \rangle =0,$ i.e $\alpha $ is $f$%
-rectifying curve.\newline
If the assertion (3) holds, from the Eq.(\ref{9.1}), we have%
\begin{equation*}
\left \langle \alpha _{f}(s),\alpha _{f}(s)\right \rangle =\left \langle
\alpha _{f},T\right \rangle ^{2}+c^{2}
\end{equation*}%
by differentiating the last equation with respect to $s$ gives%
\begin{eqnarray*}
2\left \langle \alpha _{f}(s),f(s)T\right \rangle &=&2\left \langle \alpha
_{f},T\right \rangle \left( \left \langle f(s)T,T\right \rangle +\kappa
\left \langle \alpha _{f}(s),N\right \rangle \right) \\
\left \langle \alpha _{f}(s),f(s)T\right \rangle &=&\left \langle \alpha
_{f}(s),f(s)T\right \rangle \left( 1+\frac{\kappa }{f(s)}\left \langle
\alpha _{f}(s),N\right \rangle \right)
\end{eqnarray*}%
Since $\kappa >0,$ $f(s)\neq 0$ and the norm function $\left \vert \alpha
_{f}(s)\right \vert $ is non constant function then $\left \langle \alpha
_{f}(s),N\right \rangle =0$ i.e $\alpha $ is $f$-rectifying curve.\newline
For assertion $(4)$, using the Eq.(\ref{1}), we can easily get the result.
\end{proof}

\section{Helix, rectifying curves compared to the $f$-rectifying curves}

From \cite{bm} and \cite{Cb}, any twisted curve $\alpha $ is helix if and
only if the ratio $\frac{\tau }{\kappa }$ is a nonzero constant, and it is
congruent to a rectifying curve if and only if the ratio $\left( \frac{\tau
}{\kappa }\right) ^{\prime }$ is a non null constant. How about the $f$%
-rectifying curve case?

The characterization in the terms of the ratio $\frac{\tau }{\kappa }$ is
given in the following theorem.

\begin{theorem}
\label{T2}Let $\alpha :I\subset \mathbb{R}\rightarrow \mathbb{E}^{3}$ be a
curve with strictly positive curvature $\kappa $. The curve $\alpha $ is
congruent to a $f$-rectifying curve if and only if the ratio of torsion and
curvature of the curve is%
\begin{equation*}
\frac{\tau }{\kappa }=\overline{\mu }F(s)
\end{equation*}%
where $F$ is the primitive of $f$ and $\overline{\mu }$\ is non null
constant.\newline
Moreover, if \newline
\textbf{i.} $f\equiv 0,$ we have the helix condition for $\alpha $, i.e. $%
\frac{\tau }{\kappa }$ is non null constant$,$\newline
\textbf{ii.} $f$ is a non null constant function$,$ $\alpha $ is congruent
to a rectifying curve, i.e. $\left( \frac{\tau }{\kappa }\right) ^{\prime }$
is non null constant,\newline
\textbf{iii.} $f$ is a $n$-degree polynomial, then $\alpha $ has a
characterization $\left( \frac{\tau }{\kappa }\right) ^{(n+1)}$ is non null
constant.
\end{theorem}

\begin{proof}
Let $\alpha :I\subset \mathbb{R}\rightarrow \mathbb{E}^{3}$ be a curve with
strictly positive curvature $\kappa .$\newline
If $\alpha $ is $f$-rectifying curve and using The Eq.(\ref{9}) then%
\begin{equation*}
\frac{\tau }{\kappa }=\frac{\lambda (s)}{\mu }=\frac{1}{\mu }F(s)=\overline{%
\mu }F(s)
\end{equation*}%
Hence, the ratio of torsion and curvature of the curve $\alpha $ satisfied
the assertions $(i)$, $(ii)$ and $(iii)$ according to the values of the
function $f$.\newline
Conversely, $\alpha :I\subset \mathbb{R}\rightarrow \mathbb{E}^{3}$ be a
curve with positive curvature $\kappa $ such that
\begin{equation*}
\frac{\tau }{\kappa }=\overline{\mu }F(s)
\end{equation*}%
by using the Frenet-Serret equations given in Eq.(\ref{1}), we get%
\begin{equation*}
\frac{d}{ds}(\int f(s)d\alpha -F(s)\ T-\frac{1}{\overline{\mu }}\ B)=0
\end{equation*}%
which conclude that $\alpha $ is congruent to a $f$-rectifying curve.
\end{proof}

\subsection{Physical interpretation}

In mechanics terms and from \cite{Cb}, up to rigid motions, the general
helix and the rectifying curves are characterized as those curves that are
in equilibrium under the action of the force field
\begin{equation*}
\mathbf{F}=f(s)T-\tau N
\end{equation*}%
for $f(s)=0$ and nonzero constant $f(s)=f,$ respectively.\newline
If $f(s)$ is non constant function and using the Theorem \ref{T2}, the curve
is not rectifying curve.

In the case when $f$ is non constant function, the $f$-rectifying curves are
characterized as those curves that are in equilibrium under the action of
the force field $\mathbf{F}$ for non rigid motions.

Now, up to rigid motions, for the action of the force field $\mathbf{\bar{F}}%
=-\frac{\tau }{F}N$ where $F$ is the primitive function of $f$, here the
curve is in equilibrium when it is $f$-rectifying curve. ( For more detail
for equilibrium curves definition, see \cite{ap} and \cite{ay})

\section{Classification of $f$-rectifying curves}

We determine in the following theorem explicitly all $f$-rectifying curves
where $f$ is non null function by dilating a vector $Y$ in $\mathbb{S}^{2}$,
with a distance function $\rho .$

\begin{theorem}
\label{T3}Let $\alpha :I\subset \mathbb{R}\rightarrow \mathbb{E}^{3}$ be a
curve with $\kappa >0$ and $f$ be an integrable non null function. Then $%
\alpha $ is a $f$-rectifying curve if and only if, up to parametrization, it
is given by%
\begin{equation}
\alpha (t)=\tfrac{c\sec (t+t_{0})}{f(F^{-1}(c\tan (t+t_{0})))}Y(t)-c^{2}\int
f^{^{\prime }}(F^{-1}(c\tan (t+t_{0})))\left( \tfrac{\sec (t+t_{0})}{%
f(F^{-1}(c\tan (t+t_{0})))}\right) ^{3}Y(t)dt  \label{10}
\end{equation}%
where $c$ is a strictly positive number, $F$ a primitive function of $f$
with $F(0)=c\tan t_{0},$ and $Y(t)$ is a curve in $\mathbb{S}^{2}$ and not
an arc of the great circle.
\end{theorem}

\begin{proof}
Let $\alpha :I\subset \mathbb{R}\rightarrow \mathbb{E}^{3}$ be a $f$%
-rectifying, without loss of generality, we suppose that $\alpha $ is a unit
speed curve with $\kappa >0$ and let us define a unit vector by%
\begin{equation}
Y(s)=\frac{\alpha _{f}(s)}{\rho (s)}  \label{10.1}
\end{equation}%
Using the Theorem \ref{T1} (1), we have%
\begin{equation}
\alpha _{f}(s)=\int f(s)d\alpha =\int f(s)T(s)ds=\sqrt{F^{2}(s)+c^{2}}Y(s)
\label{11}
\end{equation}%
by derivating the Eq.(\ref{11}) with respect to $s$ and taking account that $%
Y(s)$ is orthogonal to $Y^{\prime }(s)$, we get%
\begin{equation*}
f(s)T(s)=\frac{f(s)F(s)}{\sqrt{F^{2}(s)+c^{2}}}Y(s)+\sqrt{F^{2}(s)+c^{2}}%
Y^{\prime }(s)
\end{equation*}%
and%
\begin{equation}
f^{2}(s)=\frac{f^{2}(s)F^{2}(s)}{F^{2}(s)+c^{2}}+\left(
F^{2}(s)+c^{2}\right) r^{2}  \label{12}
\end{equation}%
where $r$ denote the norm of $Y^{\prime }$, we suppose that $c$ and $f(s)$
have a similar sign, then the Eq.(\ref{12}) turns to%
\begin{equation}
r=\frac{cf(s)}{F^{2}(s)+c^{2}}  \label{13}
\end{equation}%
Let us put%
\begin{equation*}
t=\int_{0}^{s}\frac{cf(u)}{F^{2}(u)+c^{2}}du=\arctan \left( \frac{F(s)}{c}%
\right) -\arctan \left( \frac{F(0)}{c}\right)
\end{equation*}%
then%
\begin{equation}
s=F^{-1}(c\tan (t+t_{0}))  \label{14}
\end{equation}%
where $t_{0}=\arctan \left( \frac{F(0)}{c}\right) .$\newline
Substituting the Eq.(\ref{14}) in the Eq.(\ref{11}), we have
\begin{equation}
\alpha _{f}(t)=c\sec (t+t_{0})Y(t)  \label{15}
\end{equation}%
and%
\begin{equation}
f(F^{-1}\left( c\tan (t+t_{0})\right) )T(t)=c\left( \sec (t+t_{0})\
Y(t)\right) _{t}^{\prime }  \label{15.1}
\end{equation}%
by integration by parts, we get finally%
\begin{eqnarray}
\alpha (t) &=&c\int \tfrac{\left( \sec (t+t_{0})\ Y(t)\right) ^{\prime }}{%
f(F^{-1}(c\tan (t+t_{0})))}dt  \label{15.2} \\
&=&\tfrac{c\sec (t+t_{0})}{f(F^{-1}(c\tan (t+t_{0})))}Y(t)-c\int \left( %
\left[ f(F^{-1}(c\tan (t+t_{0})))\right] ^{-1}\right) ^{\prime }\sec
(t+t_{0})\ Y(t)dt  \notag \\
&=&\tfrac{c\sec (t+t_{0})}{f(F^{-1}(c\tan (t+t_{0})))}Y(t)-c^{2}\int
f^{^{\prime }}(F^{-1}(c\tan (t+t_{0})))\left( \tfrac{\sec (t+t_{0})}{%
f(F^{-1}(c\tan (t+t_{0})))}\right) ^{3}Y(t)dt  \notag
\end{eqnarray}%
Now, let calculate the curvature function $\kappa $ of $\alpha .$\newline
The frame $\left \{ Y,Y^{\prime },Y\times Y^{\prime }\right \} $ are an
orthonormal frame in $\mathbb{E}^{3}$ of unit speed curve $Y(t)$. We have
Frenet formulas%
\begin{equation}
\left \{
\begin{array}{ccccc}
Y^{\prime } & = &  & Y^{\prime } &  \\
Y^{\prime \prime } & = & -Y & + & g(t)\ Y\times Y^{\prime } \\
\left( Y\times Y^{\prime }\right) ^{\prime } & = &  & -g(t)Y^{\prime } &
\end{array}%
\right.  \label{15.21}
\end{equation}%
then the unit speed vector and the normal vector of $Y$ are%
\begin{equation*}
t_{Y}=Y^{\prime };\  \ n_{Y}=\frac{-1}{\kappa _{Y}}Y+\frac{g(t)}{\kappa _{Y}}%
Y\times Y^{\prime }
\end{equation*}%
where $\kappa _{Y}$ is the curvature function of $Y,\ $wich give%
\begin{equation}
\kappa _{y}(t)=\sqrt{1+g^{2}(t)}  \label{15.22}
\end{equation}%
The speed curve of $\alpha $ is given by%
\begin{equation*}
\alpha ^{\prime }(t)=\frac{c\sec ^{2}(t+t_{0})}{f(F^{-1}(c\tan (t+t_{0})))}%
\left( \sin (t+t_{0})\ Y(t)+\cos (t+t_{0})\ Y^{\prime }(t)\right)
\end{equation*}%
then the speed $r_{\alpha }$ and the tangent vector $T$ are%
\begin{equation}
r_{\alpha }=\frac{c\sec ^{2}(t+t_{0})}{f(F^{-1}(c\tan (t+t_{0})))};\  \
T=\sin (t+t_{0})\ Y(t)+\cos (t+t_{0})\ Y^{\prime }(t)  \label{15.3}
\end{equation}%
From the the arc-length parameter $s$ of $\alpha $%
\begin{equation*}
\frac{ds}{dt}=\frac{c\sec ^{2}(t+t_{0})}{f(F^{-1}(c\tan (t+t_{0})))}
\end{equation*}%
we have by differentiating $T$ given in Eq.(\ref{15.3}) and Eq.(\ref{15.21})%
\begin{equation*}
\kappa \frac{c\sec ^{2}(t+t_{0})}{f(F^{-1}(c\tan (t+t_{0})))}N=\cos
(t+t_{0})g\left( t\right) y\times y^{\prime }
\end{equation*}%
using Eq.(\ref{15.22}), we get%
\begin{equation*}
\kappa =\frac{1}{c}\cos ^{3}(t+t_{0})\ f(F^{-1}(c\tan (t+t_{0})))\sqrt{%
\kappa _{y}^{2}-1}
\end{equation*}%
for $\kappa >0,$ it is necessary that $\kappa _{y}>1$ which impose that $Y$
is not an arc of the great circle in $\mathbb{S}^{2}.$\newline
Conversely, Let $\alpha $ be a curve defined by Eq.(\ref{10}). The
derivative of $f$-position vector is%
\begin{equation*}
\alpha _{f}^{\prime }=c\sec (t+t_{0})\ (\tan (t+t_{0})\ Y(t)+Y^{\prime }(t))
\end{equation*}%
from the orthogonality of $Y$ and $Y^{\prime }$, we have%
\begin{equation*}
\left \vert \alpha _{f}^{\prime }\right \vert =c\sec ^{2}(t+t_{0})
\end{equation*}%
and%
\begin{equation*}
\left \langle \alpha _{f}^{N},\alpha _{f}^{N}\right \rangle =\rho ^{2}\left(
t\right) -\frac{\left \langle \alpha _{f},\alpha _{f}{}^{\prime }\right
\rangle }{\left \vert \alpha _{f}{}^{\prime }\right \vert ^{2}}=c^{2}
\end{equation*}%
where the norm of the vector $\alpha _{f}$, $\rho $ is given by $\rho \left(
t\right) =c\sec (t+t_{0})$. Then the normal component $\alpha _{f}^{N}$ of
the $f$-position vector has constant length and using the Theorem \ref{T1}, $%
\alpha $ is a $f$-rectifying curve.
\end{proof}

\begin{corollary}
From the Eq.(\ref{15}), if $\alpha $ is $f$-rectifying curve then the curve $%
\alpha _{f}$ is rectifying curve (see \cite{Cb}).
\end{corollary}

\bigskip

We close this section with the following examples.

\begin{example}
If we put $f$ as a non null constant function in Eq.(\ref{10}), we find the
expression of the rectifying curve given in \cite{Cb}.
\end{example}

\begin{example}
Let's take $Y(t)=\frac{1}{\sqrt{2}}\left( \sin \sqrt{2}t,\cos \sqrt{2}%
t,1\right) $ a unit speed curve in $\mathbb{S}^{2}.$\newline
Let $f$ be an integrable function given by $f(t)=\sec ^{2}t$, its primitive
function is$\ F(t)=\tan t$ (here we take $t_{0}=0$ and $c=1$)$,$ with the
inverse $F^{-1}(t)=\arctan t.$\newline
Substituting the values of $f$ and $F$ in Eq.(\ref{10}) and by an
integration calculation, the curve $\alpha $ defined by
\begin{equation*}
\alpha (t)=\left( 2\sin t\cos \sqrt{2}t-\tfrac{1}{\sqrt{2}}\sin \sqrt{2}%
t\cos t,-2\sin t\sin \sqrt{2}t-\tfrac{1}{\sqrt{2}}\cos t\cos \sqrt{2}t,%
\tfrac{3}{\sqrt{2}}\cos t\right)
\end{equation*}%
is a $f$-rectifying curve for $f(t)=\sec ^{2}t$.
\begin{figure}[h]
\centering
\includegraphics[width=2in,height=1.8in]{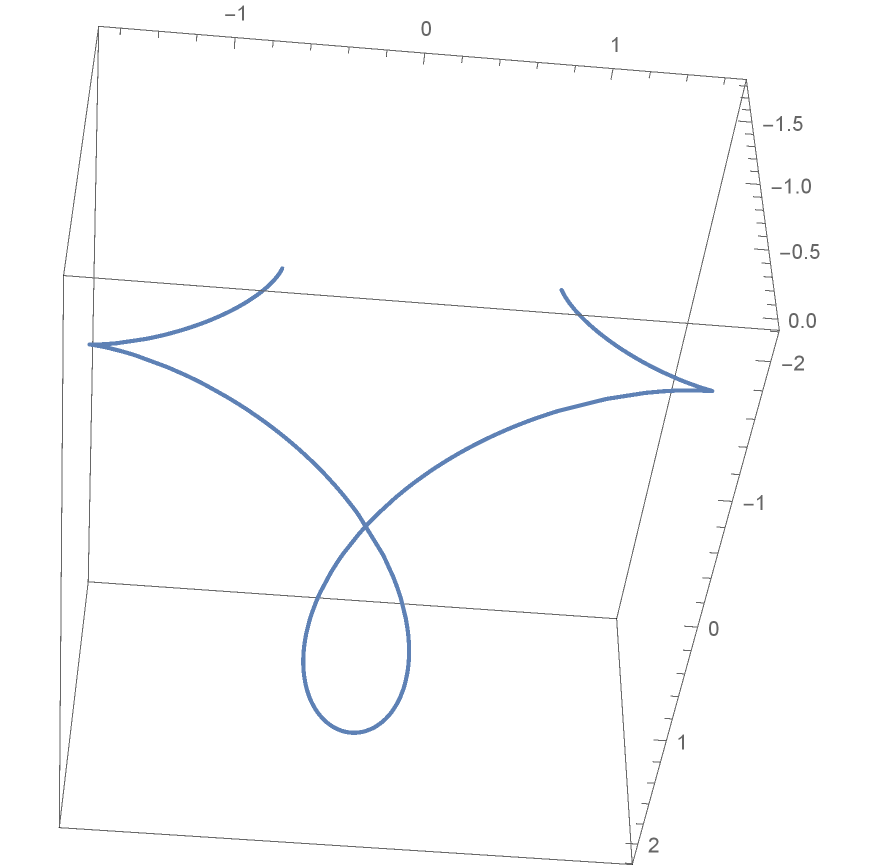} %
\includegraphics[width=2in,height=1.8in]{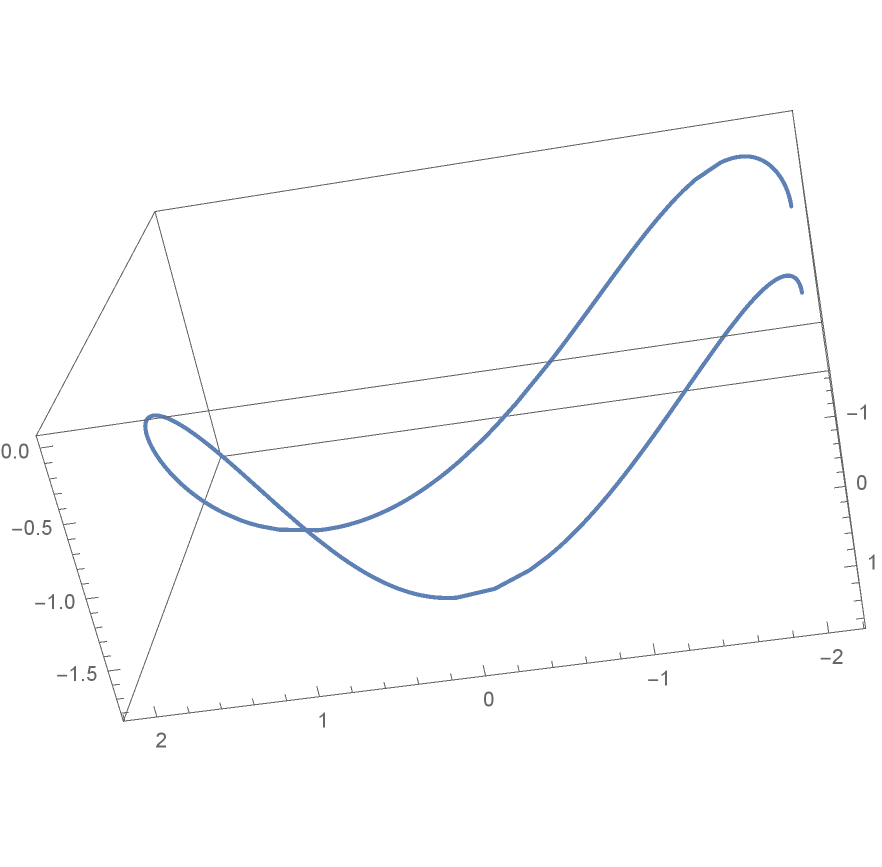}
\caption{\textmd{The $f$-rectifying curve $\protect \alpha $ with $f(t)=\sec
^{2}t$.}}
\end{figure}
\end{example}

\begin{example}
Let's take $Y(t)=\frac{1}{\sqrt{2}}\left( \sin t,\sin t,\sqrt{2}\cos
t\right) \in \mathbb{S}^{2}$ and let $f(t)=2t$ where its primitive function
is$\ F(t)=t^{2}$ (here we take $t_{0}=0$ and $c=1$)$,$ with the inverse $%
F^{-1}(t)=\sqrt{\left \vert t\right \vert }.$\newline
Substituting the values of $f$ and $F$ in Eq.(\ref{10})%
\begin{eqnarray*}
\alpha (t) &=&\tfrac{\sec t}{2\sqrt{\tan t}}Y(t)-\int \left( \frac{1}{2\sqrt{%
\tan t}}\right) ^{\prime }\sec (t)Y(t)dt \\
&=&\left(
\begin{array}{c}
\sqrt{2}\sqrt{\tan t}+\sqrt{\frac{2}{\cot 2t+\sec 2t}} \\
\sqrt{2}\sqrt{\tan t}+\sqrt{\frac{2}{\cot 2t+\sec 2t}} \\
\frac{1}{2}\sqrt{\cot t}-\frac{1}{2}\sqrt{\cot 2t+\sec 2t}%
\end{array}%
\right)
\end{eqnarray*}%
is a $f$-rectifying curve for $f(t)=2t$.
\begin{figure}[h]
\centering
\includegraphics[width=2in,height=1.8in]{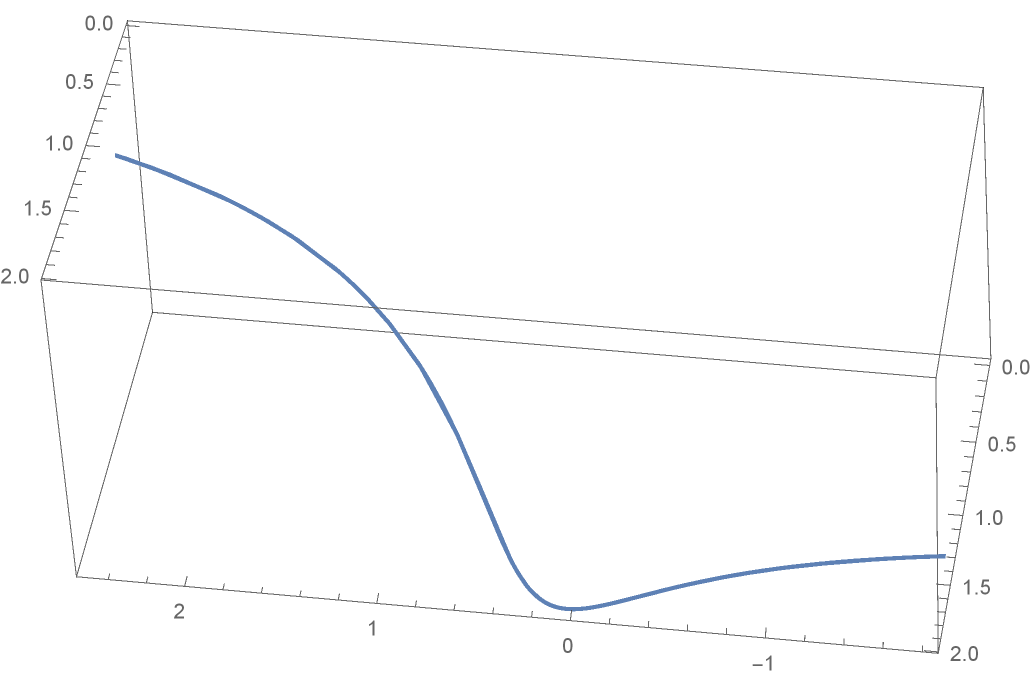} %
\includegraphics[width=2in,height=1.8in]{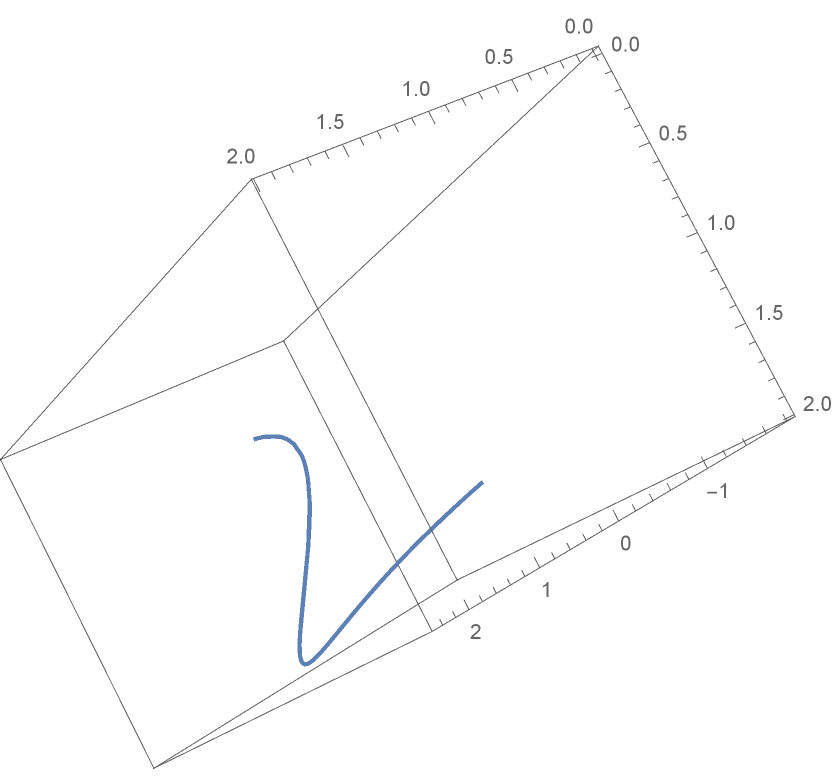}
\caption{\textmd{The $f$-rectifying curve $\protect \alpha $ with $f(t)=2t$.}}
\end{figure}
\end{example}

\begin{example}
For vector $Y(t)=\frac{1}{\sqrt{2}}\left( \sin t,\sin t,\sqrt{2}\cos
t\right) $ in $\mathbb{S}^{2}.$ Let $f$ be an integrable function given by $%
f(t)=e^{t}$, its primitive function is$\ F(t)=e^{t}$ (here we take $c_{1}=0$
and $c=1$)$,$ with the inverse $F^{-1}(t)=\ln \left \vert t\right \vert .$%
\newline
Substituting the values of $f$ and $F$ in Eq.(\ref{10}), the curve $\alpha $%
\begin{eqnarray*}
\alpha (t) &=&\tfrac{\sec t}{\tan t}Y(t)-\int \left( \cot t\right) ^{\prime
}\sec (t)\ Y(t)dt \\
&=&\left(
\begin{array}{c}
\frac{1}{\sqrt{2}}+\frac{1}{2\sqrt{2}}\left( \ln \left( \frac{1-\cos t}{%
1+\cos t}\right) +2\csc t\right) \\
\frac{1}{\sqrt{2}}+\frac{1}{2\sqrt{2}}\left( \ln \left( \frac{1-\cos t}{%
1+\cos t}\right) +2\csc t\right) \\
\cot t-\frac{1}{2}\left( \ln \left( \frac{1-\sin t}{1+\sin t}\right) +2\sec
t\right)%
\end{array}%
\right)
\end{eqnarray*}%
is a $f$-rectifying curve for $f(t)=e^{t}$.
\begin{figure}[h]
\centering
\includegraphics[width=2in,height=1.8in]{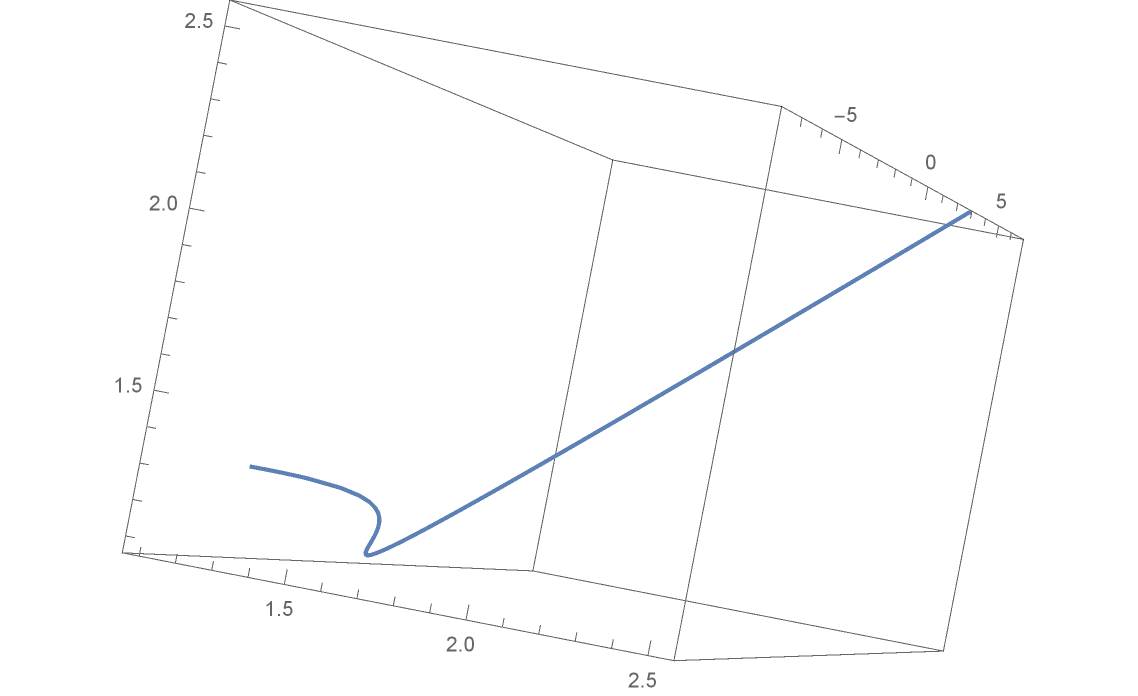} %
\includegraphics[width=2in,height=1.8in]{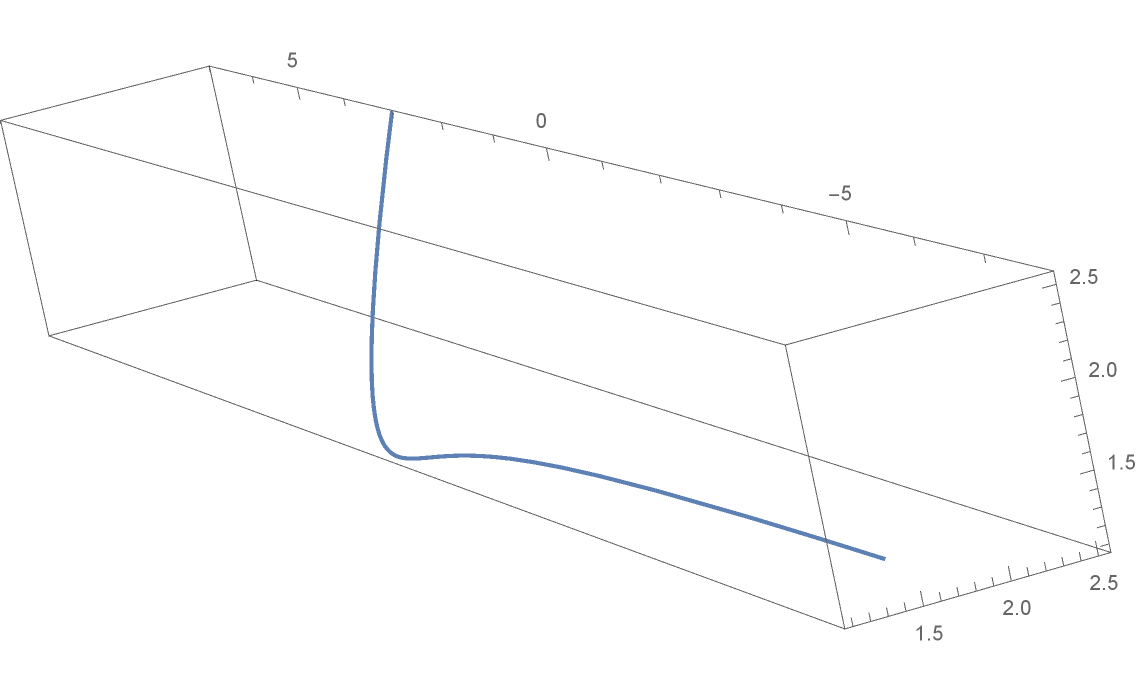}
\caption{\textmd{The $f$-rectifying curve $\protect \alpha $ with $f(t)=e^{t}$%
.}}
\end{figure}
\end{example}

\end{document}